\documentclass{article}

\usepackage{amsmath,amsthm,amssymb}
\usepackage{mathrsfs,bm}
\usepackage{mathtools, enumitem}

\newcommand{\jbk}[1]{\left\langle {#1} \right\rangle}

\newcommand{\rmop}[1]{\mathop{\mathrm{#1}}}

\newcommand{\df}{\mathrm{d}}

\newcommand{\lap}{\Delta}

\newcommand{\nv}{\nu}

\newcommand{\Ker}{\rmop{Ker}}
\newcommand{\Ran}{\rmop{Ran}}

\newcommand{\p}{\partial}

\newcommand{\Rbb}{\mathbb{R}}

\newcommand{\Hscr}{\mathscr{H}}

\newcommand{\Acal}{\mathcal{A}}

\newcommand{\Dcal}{\mathcal{D}}

\newcommand{\Gcal}{\mathcal{G}}
\newcommand{\Hcal}{\mathcal{H}}

\newcommand{\Jcal}{\mathcal{J}}
\newcommand{\Kcal}{\mathcal{K}}

\newcommand{\Scal}{\mathcal{S}}

\newcommand{\Xcal}{\mathcal{X}}

\newcommand{\Gvf}{\varphi}

\newcommand{\Go}{\omega}

\newcommand{\GG}{\Gamma}

\newcommand{\GO}{\Omega}

\usepackage{expl3}
\ExplSyntaxOn
\cs_new_eq:NN \Repeat \prg_replicate:nn
\ExplSyntaxOff

\newtheorem{prop}{Proposition}[section]
\newtheorem{theo}[prop]{Theorem}

\newtheorem{lemm}[prop]{Lemma}

\theoremstyle{definition}

\newtheorem*{note*}{Note}
\newtheorem*{claim*}{Claim}
\newtheorem*{exam*}{Example}
\newtheorem*{rema*}{Remark}
\newtheorem*{exer*}{Exercise}
\newtheorem*{prob*}{Problem}

\numberwithin{equation}{section}

\usepackage{hyperref}
\hypersetup{%
colorlinks=true,
linkcolor=blue,
urlcolor=black,
}

\mathtoolsset{showonlyrefs=false}

\newcommand{\SL}{\Scal}
\newcommand{\DL}{\Dcal}
\newcommand{\NP}{\Kcal}

\newcommand{\Hsb}{\Hcal}

\newcommand{\Hdf}{\Hsb_\mathrm{div}}

\newcommand{\Hdrf}{\Hsb_\mathrm{div, rot}}

\newcommand{\divf}{\mathrm{div}}

\newcommand{\Mpm}{\nabla \Hcal}

\title{Decompositions of surface vector fields and topological characterizations of the codimensions\thanks{\footnotesize This work was supported by NRF (of S. Korea) grant 2022R1A2B5B01001445.}}

\author{Shota Fukushima\thanks{Department of Mathematics and Institute of Applied Mathematics, Inha University, Incheon 22212, S. Korea. Email: \texttt{shota.fukushima.math@gmail.com, hbkang@inha.ac.kr}.} \thanks{Corresponding author.}
\and Hyeonbae Kang\footnotemark[2]}

\begin{document}
\maketitle

\begin{abstract}
We prove that the space of vector fields on the boundary of a bounded domain in three dimensions is decomposed into three subspaces orthogonal to each other: elements of the first one extend to the inside of the domain as gradient fields of harmonic functions, the second one to the outside of the domain as gradient fields of harmonic functions, and the third one to both the inside and the outside as divergence-free harmonic vector fields. We also characterize by the first Betti numbers the codimensions of inclusions of various subspaces related to the decomposition.
\end{abstract}




\section{Introduction}

In this paper, we deal with decompositions of vector fields defined on the boundary $\p\GO$ of a bounded domain $\GO$ in $\Rbb^3$. We prove that the space of vector fields on $\p\GO$ is decomposed into three orthogonal subspaces: elements of the first one extend to $\GO$ as gradient fields of harmonic functions, the second one to $\Rbb^3 \setminus \overline{\GO}$ as gradient fields of harmonic functions (with a proper decay at $\infty$), and the third one to both $\GO$ and $\Rbb^3 \setminus \overline{\GO}$ as divergence-free harmonic vector fields. Another decomposition is obtained in \cite{FJKc1a}: the space of vector fields on $\p\GO$ is decomposed into three orthogonal subspaces: elements of the first one extend to $\GO$ as divergence-free, rotation-free vector fields, the second one to $\Rbb^3 \setminus \overline{\GO}$ as divergence-free, rotation-free vector fields, and the third one to both $\GO$ and $\Rbb^3 \setminus \overline{\GO}$ as divergence-free harmonic vector fields which can be represented in terms of the double layer potential type operator (see \eqref{divdouble}). These two decompositions are the same if $\p\GO$ is simply connected, and they are different if $\p\GO$ is not simply connected. This will be proved by characterizing the codimensions of the inclusion relations among various subspaces involved in decompositions in terms of the first Betti numbers of $\GO$, $\Rbb^3 \setminus \overline{\GO}$, and $\p\GO$.

Let us now present results in a more precise manner.
We begin with introducing the inner product which is used in describing the orthogonality of the decompositions.

Let $\Omega\subset \Rbb^3$ be a bounded domain with the Lipschitz boundary $\p\GO$. Let $\GG(x)$ be the fundamental solution to the Laplace operator, namely, $\GG(x)=- (4\pi |x|)^{-1}$, and let $\SL$ be the single layer potential for the Laplace operator, namely,
\begin{equation}\label{eq_sl}
    {\SL}[f](x)=\int_{\partial\Omega}\GG(x-y)f(y)\, \df \sigma (y).
\end{equation}
Note that ${\SL}[f](x)$ is defined for $x \in \Rbb^3$. If we confine $x$ on $\p\GO$, then $\SL$ is an invertible operator from $H^{-1/2}(\partial\Omega)$ onto $H^{1/2}(\partial\Omega)$  \cite{Verchota84}. Here and throughout this paper, $H^s (X, \Rbb^N)$ stands for the $L^2$-Sobolev space of order $s$ of $\Rbb^N$-valued functions on an appropriate space $X$, and we use the shorthanded notation $H^s (X, \Rbb)=H^s(X)$. The definition of ${\SL}[f]$ for $\Rbb^3$-valued function is obvious, and $\SL$ is invertible from $H^{-1/2}(\partial\Omega, \Rbb^3)$ onto $H^{1/2}(\partial\Omega, \Rbb^3)$. Denoting its inverse by $\Scal^{-1}$, the bilinear form
\begin{equation}\label{eq_inner_product}
    \jbk{f, g}_*:=-\jbk{\SL^{-1}[f], g}_{\partial\Omega}
\end{equation}
defines an inner product on $H^{1/2}(\partial\Omega, \Rbb^3)$ whose associated norm is equivalent to the Sobolev norm. Here $\jbk{\cdot, \cdot}_{\partial\Omega}$ stands for the dual pairing of $H^{-1/2}(\partial\Omega, \Rbb^3)$ and $H^{1/2}(\partial\Omega, \Rbb^3)$. This inner product was introduced in \cite{KPS} for symmetrization of the Neumann-Poincar\'e operator. 

The inner product \eqref{eq_inner_product} is closely related to interior and exterior boundary value problems of the Laplace equations.
For $f\in H^{1/2}(\partial\Omega, \Rbb^3)$, let $v_-^f\in H^1(\Omega, \Rbb^3)$ be the unique solution to the problem
\begin{equation}
    \label{eq_bvp_interior}
    \begin{cases}
        \lap v_-^f=0 & \text{in } \Omega, \\
        v_-^f=f & \text{on } \partial\Omega.
    \end{cases}
\end{equation}
For the exterior problem, it is convenient to use notation $\GO^+:= \Rbb^3\setminus \overline{\Omega}$. $\Omega^-$ denotes $\Omega$. Let $H^s_{-1}(\GO^+, \Rbb^3)$ be the weighted Sobolev space
\begin{equation}\label{eq_weighted_Sobolev}
    H^s_{-1}(\GO^+, \Rbb^3):=\left\{ u\in L_\mathrm{loc}^1 (\GO^+, \Rbb^3)\,\middle|\,
    \begin{aligned}
        &(1+|x|)^{-1}|\nabla^j u|\in L^2 (\GO^+, \Rbb^3) \\
        &\text{for all } 0\leq j\leq s
    \end{aligned}
    \right\}
\end{equation}
for a nonnegative integer $s$. Let $v_+^f\in H_{-1}^1 (\GO^+, \Rbb^3)$ be the unique solution to the exterior problem
\begin{equation}
    \label{eq_bvp_exterior}
    \begin{cases}
        \lap v_+^f=0 & \text{in } \GO^+, \\
        v_+^f=f & \text{on } \partial\Omega.
    \end{cases}
\end{equation}
We see that
\begin{equation}\label{SS^-}
v_\pm^f(x) = \Scal [\Scal^{-1}[f]](x), \quad x \in \GO^\pm.
\end{equation}
It then follows from the divergence theorem and the jump formula enjoyed by $\Scal$ that the following relation holds:
\[
    \jbk{f, g}_*=\jbk{ \nabla v_-^f, \nabla v_-^g}_{L^2 (\Omega^-, \Rbb^{3\times 3})}+\jbk{ \nabla v_+^f, \nabla v_+^g}_{L^2 (\Omega^+, \Rbb^{3\times 3})}.
\]
The jump formula is
$$
\p_\nu \Scal[\Gvf]|_+ - \p_\nu \Scal[\Gvf]|_- =\Gvf \quad\text{on } \p\GO.
$$
Here and throughout this paper, $\nu$ denotes the normal vector field on $\p\GO$ outward to $\GO$, $\p_\nu$ the normal derivative, and $u|_\pm$ the limit from $\GO^\pm$ to $\p\GO$.

We now review the results of \cite{FJKc1a} which motivate current work and are complementary to its results. For that purpose, we introduce the following subspaces of $H^{1/2}(\partial\Omega, \Rbb^3)$:
\begin{align*}
    \Hdf^\pm &:=\{ f\in H^{1/2}(\partial\Omega, \Rbb^3) \mid \nabla \cdot v_\pm^f=0 \text{ in } \Omega^\pm \}, \\
    \Hdf &:= \Hdf^- \cap \Hdf^+, \\
    \Hdrf^\pm &:=\{ f\in \Hdf^\pm \mid \nabla \times v_\pm^f=0 \text{ in } \Omega^\pm \}.
\end{align*}
One can easily see that elements of $\Hdf$ extend to both $\GO^-$ and $\GO^+$ as divergence-free harmonic vector fields; those of $\Hdrf^-$ and $\Hdrf^+$ respectively extend to $\GO^-$ and $\GO^+$ as divergence-free, rotation-free vector fields.

The following operator of the double layer potential type is introduced in \cite{FJKc1a}:
\begin{align}
    \DL^\divf [f](x):=&
    \int_{\partial\Omega} \nabla_y \GG(x-y)(\nv_y\cdot f(y))- \nv_y (\nabla_y \GG(x-y)\cdot f(y)) \, \df \sigma (y) \nonumber \\
    & + \int_{\partial\Omega} (\nabla_y \GG(x-y) \cdot \nv_y) f(y)\, \df \sigma (y), \quad x\in \Rbb^3\setminus \partial\Omega, \label{divdouble}
\end{align}
for $f\in H^{1/2}(\partial\Omega, \Rbb^3)$. We mention that the second operator on the right hand side of \eqref{divdouble} is the standard double layer potential for the Laplacian. It is shown that
\begin{equation}\label{div-free-double}
\lap \DL^\divf [f]=0, \quad \nabla\cdot \DL^\divf [f]=0 \quad \text{in } \Rbb^3\setminus \partial \Omega.
\end{equation}
Because of these properties, $\DL^\divf$ is called the divergence-free double layer potential. Let $\NP^\divf$ be the corresponding Neumann-Poincar\'e operator, that is,
\begin{align*}
    \NP^\divf [f](x)= &
    \rmop{p.v.} \int_{\p\GO} \nabla_y \GG(x-y)(\nv_y\cdot f(y))- \nv_y (\nabla_y \GG(x-y)\cdot f(y)) \, \df \sigma (y) \nonumber \\
    & + \rmop{p.v.} \int_{\p\GO} (\nabla_y \GG(x-y) \cdot \nv_y) f(y)\, \df \sigma (y), \quad x\in \p\GO,
\end{align*}
where $\rmop{p.v.}$ stands for the Cauchy principal value.
Then the following jump relation holds \cite[Proposition 2.4]{FJKc1a}:
\[
\DL^\divf [f]|_{\pm}(x)= \left(\NP^\divf \mp \frac{1}{2}I \right)[f](x), \quad x\in \p\GO.
\]
It is proved in \cite[Theorem 2.9]{FJKc1a} that
\[
\Hdrf^\pm = \Ker \left( \NP^\divf \pm \frac{1}{2}I \right).
\]
Since $\NP^\divf$ is self-adjoint on $H^{1/2}(\p\GO, \Rbb^3)$ equipped with the inner product \eqref{eq_inner_product},
the Hilbert space $H^{1/2}(\partial\Omega, \Rbb^3)$ admits the following orthogonal decomposition:
\begin{equation}
            \label{eq_vf_decomposition}
            H^{1/2}(\partial\Omega, \Rbb^3)=\Hdrf^-\oplus \Hdrf^+ \oplus \overline{\Hdf^\Dcal},
        \end{equation}
where
\begin{equation}
    \label{eq_X}
    \Hdf^\Dcal := \Ran \left( \NP^\divf + \frac{1}{2}I \right)\left( \NP^\divf - \frac{1}{2}I \right).
\end{equation}
Furthermore, if $\partial\Omega$ is $C^{1, \alpha}$ for some $\alpha>1/2$, then $\Hdf^\Dcal$ is proved to be a closed subspace \cite[Theorem 5.4]{FJKc1a} and we obtain the following orthogonal decomposition:
\begin{equation}
    \label{eq_vf_decomposition2}
    H^{1/2}(\partial\Omega, \Rbb^3)=\Hdrf^-\oplus \Hdrf^+ \oplus \Hdf^\Dcal.
\end{equation}

Note that if $f \in \Hdf^\Dcal$, then there exists $f_0 \in H^{1/2}(\partial\Omega, \Rbb^3)$ such that
\[
f=\left.\DL^\divf \left( \NP^\divf - \frac{1}{2}I \right)[f_0]\right|_-=\left.\DL^\divf \left( \NP^\divf + \frac{1}{2}I \right)[f_0]\right|_+.
\]
In view of \eqref{div-free-double}, these formulas imply that $f \in \Hdf$. Thus, we have
\begin{equation}\label{eq:subset}
\Hdf^\Dcal \subset \Hdf.
\end{equation}
An upper bound of the codimension of this inclusion relation is obtained using an argument of de Rham cohomology group \cite[Theorem 1.2]{FJKc1a}: If $\partial\Omega$ is $C^{1, \alpha}$ for some $\alpha>1/2$, then
        \begin{equation}
            \label{eq_vf_betti}
            \dim (\Hdf / \Hdf^\Dcal)\leq b_1 (\partial\Omega)
        \end{equation}
        where $b_1 (\partial\Omega)$ is the first Betti number of $\p\GO$. In particular, if $\partial\Omega$ is simply connected, namely, $b_1 (\partial\Omega)=0$, then
        \begin{equation}
            \label{Dcal=}
            \Hdf^\Dcal = \Hdf,
        \end{equation}
and the following orthogonal decomposition with respect to the inner product \eqref{eq_inner_product} holds:
\begin{equation}
            \label{eq_vf_decomposition3}
            H^{1/2}(\partial\Omega, \Rbb^3)=\Hdrf^-\oplus \Hdrf^+ \oplus \Hdf.
        \end{equation}

The purpose of this work is to investigate further the relation \eqref{eq:subset} and the inequality \eqref{eq_vf_betti}. We aim to characterize the orthogonal complement $\Hdf^\perp$ of $\Hdf$ and to prove the equality in \eqref{eq_vf_betti}. These two questions are closely related through the following relation which holds by \eqref{eq_vf_decomposition2}:
\[
\dim (\Hdf / \Hdf^\Dcal)= \dim ((\Hdf^\Dcal)^\perp/\Hdf^\perp) = \dim ((\Hdrf^-+\Hdrf^+)/\Hdf^\perp).
\]

Since $\Hdf = \Hdf^- \cap \Hdf^+$ and $\Hdf^\perp$ are closed, we have
    \begin{equation}\label{Hdfperp}
        \Hdf^\perp =\overline{(\Hdf^-)^\perp +(\Hdf^+)^\perp}.
    \end{equation}
So we seek to characterize the subspaces $(\Hdf^\pm)^\perp$. It turns out that they are characterized by the following subspaces:
\begin{align}
    \Mpm^-&:=\{ (\nabla u)|_- \mid u \in H^2 (\Omega^-) \text{ is harmonic in } \GO^- \}, \label{eq_mm_defi} \\
    \Mpm^+&:=\{ (\nabla u)|_+ \mid u \in H^2_{-1}(\GO^+) \text{ is harmonic in } \GO^+ \} . \label{eq_mp_defi}
\end{align}
We emphasize that $\Mpm^\pm$ is the subspace of $\Hdrf^\pm$ consisting of vector fields which can be extended as gradients of harmonic functions.

The following theorem is the crux of this paper.

\begin{theo}\label{1001}
If $\Omega\subset \Rbb^3$ is a bounded $C^{1, \alpha}$-domain for some $\alpha>1/2$, then it holds that
\begin{equation}\label{1000}
(\Hdf^\pm)^\perp=\Mpm^\mp.
\end{equation}
\end{theo}

Since $\Hdrf^- \perp \Hdrf^+$, we have $\Mpm^- \perp \Mpm^+$. It then follows from \eqref{Hdfperp} and Theorem \ref{1001} that
\[
\Hdf^\perp = \Mpm^- \oplus \Mpm^+.
\]
Thus we obtain the following theorem yielding a new decomposition of surface vector fields. There \eqref{eq_vf_potential2} is a consequence of \eqref{Dcal=}.
\begin{theo}\label{theo_vf_potential}
    If $\Omega\subset \Rbb^3$ is a bounded $C^{1, \alpha}$-domain for some $\alpha>1/2$, then we have the decomposition
    \begin{equation}
        \label{eq_vf_potential}
        H^{1/2}(\partial\Omega, \Rbb^3)=\Mpm^- \oplus \Mpm^+\oplus \Hdf
    \end{equation}
    which is orthogonal with respect to the inner product \eqref{eq_inner_product}. Moreover, if $\p\GO$ is simply connected, then we have
    \begin{equation}
        \label{eq_vf_potential2}
        H^{1/2}(\partial\Omega, \Rbb^3)=\Mpm^- \oplus \Mpm^+\oplus \Hdf^\Dcal.
    \end{equation}
\end{theo}

As mentioned before, the following inclusion relations among various subspaces involved in the decompositions \eqref{eq_vf_decomposition2} and \eqref{eq_vf_potential} hold:
\begin{equation}\label{inclusionrel}
            \Mpm^\pm \subset \Hdrf^\pm, \quad \Hdf^\Dcal \subset \Hdf.
        \end{equation}
The codimensions of the inclusions are characterized in terms of the first Betti numbers as the following theorem shows.
\begin{theo}
    \label{theo_betti_from_below}
    If $\Omega \subset \Rbb^3$ is a bounded $C^{1, \alpha}$-domain for some $\alpha>1/2$, then we have
    \begin{equation}
        \label{eq_betti_from_below1}
        \dim (\Hdrf^\pm / \Mpm^\pm)= b_1 (\GO^\pm)
    \end{equation}
    and
    \begin{equation}
        \label{eq_betti_from_below}
        \dim (\Hdf / \Hdf^\Dcal)= b_1 (\partial\Omega).
    \end{equation}
    In particular, two decompositions \eqref{eq_vf_decomposition2} and \eqref{eq_vf_potential} are identical if and only if $\partial\Omega$ is simply connected.
\end{theo}

We prove Theorem \ref{theo_betti_from_below} using results of \cite{Osterbrink-Pauly20, Picard82}.  In \cite{Osterbrink-Pauly20} the quotient of the space of rotation-free fields over that of gradient fields is characterized in the interior and exterior domains using Helmholtz(-Weyl) type decomposition. In \cite{Picard82}, a Hodge-de Rham type theorem is used to show the dimension of the quotient space equals to the first Betti number.  These results are summarized in Theorem \ref{theo_Helmholtz} of this paper. We isomorphically identify the spaces $\Hdrf^\pm / \Mpm^\pm$ with those quotient spaces. The quotient space of rotation-free fields over gradient fields is reminiscent of the first de Rham cohomology group which is defined by
\[
    H_\mathrm{dR}^1 (\GO^\pm):=\frac{\Ker (\nabla\times: C^\infty (\GO^\pm, \Rbb^3)\longrightarrow C^\infty (\GO^\pm, \Rbb^3))}{\Ran (\nabla: C^\infty (\GO^\pm)\longrightarrow C^\infty (\GO^\pm, \Rbb^3))}
\]
(see, for example, \cite{Bott-Tu82}).
In \cite{FJKc1a}, the first de Rham cohomology group is used to prove \eqref{eq_vf_betti}.

It is helpful to mention that the Helmholtz decomposition is generalized to compact manifolds with boundaries in terms of the differential forms. This decomposition is called the Hodge decomposition or Hodge-Morrey-Friedrichs decomposition \cite{GMM11, Karpukhin19, Schwarz95}. For the Helmholtz decomposition in the exterior domains, refer to recent work \cite{HKSSY212D, HKSSY213D} and the references therein.

The rest of the paper is devoted to proofs of Theorems \ref{1001} and \ref{theo_betti_from_below}. They are given in Section \ref{subs_BIO} and Section \ref{subs_Betti}, respectively. This paper ends with a short discussion. The results of this paper hold for domains in $\Rbb^d$ with $d>3$ if the boundary is $C^{1,1}$. We explain this in Discussion.

\section{Proof of Theorem \ref{1001}}\label{subs_BIO}

We prove Theorem \ref{theo_vf_potential} using the layer potential theory. We define the integral operator $\Jcal$ by
\begin{align}
    \label{eq_operator_J}
    \Jcal [f](x):=-\rmop{p.v.}\int_{\partial\Omega} \nabla_y \GG(x-y)\cdot f(y)\, \df\sigma (y), \quad x \in \p\GO
\end{align}
for $f \in L^2(\partial\Omega,\Rbb^3)$. Then, its $L^2$-adjoint operator $\Jcal^*$ is given by
\begin{align}
   \label{eq_operator_Js}
    \Jcal^* [\varphi](x):=-\rmop{p.v.}\int_{\partial\Omega} \nabla_x \GG(x-y)\varphi(y)\, \df\sigma (y), \quad x \in \p\GO
\end{align}
for $\varphi \in L^2(\partial\Omega)$. Then the following mapping properties of $\Jcal$ and $\Jcal^*$ are known: see, for example, \cite[Lemma 4.3]{FJKc1a} for a proof.

\begin{lemm}
    \label{lemm_J_bdd}
    Let $\partial\Omega$ be $C^{1, \alpha}$ for some $\alpha>1/2$.
    \begin{enumerate}[label=(\roman*)]
        \item \label{enum_J_bdd}$\Jcal$ is a bounded operator from $H^s (\partial\Omega, \Rbb^3)$ into $H^s (\partial\Omega)$ for all $|s|\leq 1$.
        \item \label{enum_Js_bdd}$\Jcal^*$ is a bounded operator from $H^s (\partial\Omega)$ into $H^s (\partial\Omega, \Rbb^3)$ for all $|s|\leq 1$.
    \end{enumerate}
\end{lemm}

We invoke the jump formulas
\begin{align}
    \nabla\cdot {\SL}[g]|_\pm &= \pm \frac{1}{2}g\cdot \nv+\Jcal [g], \label{eq_jump_J} \\
    \nabla {\SL}[\varphi]|_\pm &= \pm \frac{1}{2}\varphi \nv -\Jcal^* [\varphi], \label{eq_jump_Js}
\end{align}
for $g\in L^2 (\partial\Omega; \Rbb^3)$ and $\varphi\in L^2 (\partial\Omega)$ (see \cite{DKV88}). According to \cite[Corollary 6.6]{Gesztesy-Mitrea11}, the relations \eqref{eq_jump_J} and \eqref{eq_jump_Js} can be extended to  $g \in H^{-1/2}(\partial\Omega, \Rbb^3)$ and $\varphi \in H^{-1/2}(\partial\Omega)$, respectively.

We define operators $\Acal_\pm: H^{1/2}(\partial\Omega, \Rbb^3)\to H^{-1/2}(\partial\Omega)$ by
\begin{equation}
    \label{eq_apm}
        \Acal_\pm [f]:=\pm \frac{1}{2}\nv \cdot \SL^{-1}[f]+\Jcal\SL^{-1} [f] = \nabla\cdot {\SL}[\SL^{-1}[f]]|_\pm.
\end{equation}
As before, $\SL^{-1}$ is the inverse of $\SL$ as mapping from $H^{-1/2}(\partial\Omega, \Rbb^3)$ to $H^{1/2}(\partial\Omega, \Rbb^3)$. If $\partial\Omega$ is $C^{1, \alpha}$ for some $\alpha>1/2$, then the normal vector field $\nu$ is $C^{0, \alpha}$ and the mapping $\varphi\in H^{1/2}(\partial\Omega)\mapsto \varphi\nv \in H^{1/2}(\partial\Omega, \Rbb^3)$ is bounded (see \cite[Lemma 4.2]{FJKc1a}, \cite[Theorem 1 (vi)]{Mazya-Shaposhnikova05}). By the duality, the mapping $g\in H^{-1/2}(\partial\Omega, \Rbb^3)\mapsto g\cdot \nv \in H^{-1/2}(\partial\Omega)$ is also bounded.
It then follows from Lemma \ref{lemm_J_bdd} \ref{enum_J_bdd} that $\Acal_\pm$ are bounded. Moreover, we see from \eqref{SS^-} that
$$
\Acal_- [f] = \nabla \cdot v_-^f|_-, \quad \Acal_+ [f] = \nabla \cdot v_+^f|_+,
$$
of which the following theorem is an immediate consequence.
\begin{theo}
    \label{theo_apm_div_free}
    If $\partial\Omega$ is $C^{1, \alpha}$ for some $\alpha>1/2$, then we have
    \begin{equation}
    \Ker \Acal_\pm=\Hdf^\pm.
    \end{equation}
\end{theo}

Let $\Acal_\pm^*: H^{1/2}(\partial\Omega) \to H^{-1/2}(\partial\Omega, \Rbb^3)$ be the $L^2$-adjoint operators of $\Acal_\pm$, which are represented as
\begin{equation}\label{eq_aps_j}
    \Acal_\pm^* [\varphi]=\SL^{-1}\left[\pm \frac{1}{2}\varphi \nv + \Jcal^*[\varphi]\right].
\end{equation}

\begin{theo}
    \label{theo_ran_apms_closed}
    If $\partial\Omega$ is $C^{1, \alpha}$ for some $\alpha>1/2$, then $\Ran \Acal_\pm^*$ is closed in $H^{-1/2}(\partial\Omega, \Rbb^3)$.
\end{theo}

\begin{proof}
    We only prove the theorem for $\Acal_-^*$.
    Assume that the sequence $\{\Acal_-^*[\varphi_j]\}_{j=1}^\infty$ converges to $f$ in the $H^{-1/2}(\partial\Omega, \Rbb^3)$-norm topology. By \eqref{eq_aps_j}, we have
    \begin{equation}\label{eq_j_converge}
        -\frac{1}{2}\varphi_j\nv+\Jcal^* [\varphi_j] \to {\SL}[f] \ \text{ in } H^{1/2}(\partial\Omega, \Rbb^3).
    \end{equation}

    Let $\NP^*$ be the Neumann-Poincar\'e operator on $\p\GO$ for the Laplace operator, namely,
    \begin{equation}\label{eq_NP}
        \begin{aligned}
            \NP^*[\varphi](x):=&\, \int_{\partial\Omega} \p_{\nu_x} \GG(x-y) \varphi(y)\, \df \sigma (y) \\
            =&\,\frac{1}{4\pi}\int_{\partial\Omega} \frac{(x-y)\cdot \nv_x}{|x-y|^3}\varphi (y)\, \df \sigma (y),
            \quad x \in \p\GO.
        \end{aligned}
    \end{equation}
    Then, the formula \eqref{eq_operator_Js} immediately yields the relation $\NP^*[\varphi]= \nu \cdot \Jcal^*[\varphi]$. Since the normal vector field $\nu$ is $C^{0, \alpha}$ and $\alpha>1/2$, the mapping $g\in H^{1/2}(\partial\Omega, \Rbb^3)\mapsto \nv\cdot g \in H^{1/2}(\partial\Omega)$ is bounded. Thus we have
    \begin{equation}\label{eq_nps_converge}
        -\frac{1}{2}\varphi_j+\NP^* [\varphi_j]\to \nv\cdot{\SL}[f] \text{ in } H^{1/2}(\partial\Omega).
    \end{equation}

Let
\[
        \Hsb_0:=\left\{ \varphi\in H^{1/2}(\partial\Omega) \,\middle|\, \int_{\partial\Omega} \varphi\, \df \sigma =0 \right\}.
    \]
Then,  $-1/2I+\NP^*$ is invertible on $\Hsb_0$. In fact, it is known that it is invertible on $L_0^2(\partial\Omega)$ (see \cite{Verchota84}). If $h \in \Hsb_0$ and $\varphi \in L_0^2(\partial\Omega)$ satisfy $(-1/2I+\NP^*)[\varphi]=h$, then $\varphi \in \Hsb_0$ since $\NP^*$ maps $L^2(\partial\Omega)$ into $H^{1/2}(\partial\Omega)$ boundedly (see, for example, \cite[Corollary A.1 (i)]{FJKc1a} for a proof of this fact).

Since $\Ker (-1/2I+\NP^*)\subset H^{1/2}(\partial\Omega)$ is a one-dimensional space by the connectedness of $\Omega$, we can take the basis $\varphi_0$ of $\Ker (-1/2I+\NP^*)$ such that $\int_{\partial\Omega} \varphi_0\,\df \sigma=1$.

We set
\[
\psi_j:=\varphi_j-a_j\varphi_0 \quad\text{with } a_j:=\int_{\partial\Omega} \varphi_j\, \df \sigma.
\]
Then each $\psi_j$ belongs to $\Hsb_0$ and
    \begin{equation}
        \label{eq_nps_converge_decomposed}
        \left(-\frac{1}{2}I+\NP^*\right) [\psi_j]\to \nv \cdot {\SL}[f] \quad \text{in } H^{1/2}(\partial\Omega).
    \end{equation}
Thus $\{ (-1/2I+\NP^*)[\psi_j]\}_{j=1}^\infty$ forms a Cauchy sequence in $\Hcal_0$. Since $-1/2I+\NP^*$ is invertible on $\Hsb_0$, the sequence $\{\psi_j\}_{j=1}^\infty$ is also a Cauchy sequence. Thus the limit
\[
    \psi:=\lim_{j\to\infty} \psi_j \quad \text{in } H^{1/2}(\partial\Omega)
\]
exists. Since the mappings $\varphi \mapsto \varphi \nv$ and the operator $\Jcal^*$ are bounded from $H^{1/2}(\partial\Omega)$ into $H^{1/2}(\partial\Omega, \Rbb^3)$, we have
    \[
         -\frac{1}{2}\varphi_j\nv+\Jcal^* [\varphi_j] - a_j \left(-\frac{1}{2}\varphi_0 \nv+\Jcal^* [\varphi_0] \right) \to -\frac{1}{2}\psi \nv+\Jcal^* [\psi]
    \]
in $H^{1/2}(\partial\Omega, \Rbb^3)$.

    Let $\Psi:= -1/2\varphi_0\nv+\Jcal^* [\varphi_0]$. If $\Psi =0$, then \eqref{eq_j_converge} implies that
    \[
    -\frac{1}{2}\psi \nv+\Jcal^*[\psi]={\SL}[f].
    \]
    Thus we have $\Acal_-^*[\psi]=f$ by \eqref{eq_aps_j}. If $\Psi\neq 0$, then \eqref{eq_j_converge} implies that
    \[
        a_j \Psi\to {\SL}[f]+\frac{1}{2}\psi \nv-\Jcal^* [\psi] \quad  \text{in } H^{1/2}(\partial\Omega, \Rbb^3).
    \]
    In particular, $\{ a_j \}_{j=1}^\infty$ is a Cauchy sequence in $\Rbb$ and hence the limit $a:=\lim_{j\to \infty}a_j$ exists. Thus we have
    \[
        -\frac{1}{2}(a\varphi_0+\psi)\nv+\Jcal^*[a\varphi_0+\psi]={\SL}[f],
    \]
    which is equivalent to $\Acal_-^*[a\varphi_0+\psi]=f$. Therefore $\Ran \Acal_-^*$ is closed in $H^{-1/2}(\partial\Omega, \Rbb^3)$.
\end{proof}

Suppose that $\partial\Omega$ is $C^{1, \alpha}$ for some $\alpha>1/2$. If $u \in H^2_{-1}(\GO^+)$ and $u$ is harmonic in $\GO^+$, then the trace $u|_{+}$ of $u$ belongs to $H^{3/2}(\partial\Omega)$. Since $\SL$ is an invertible operator from $H^{1/2}(\partial\Omega)$ onto $H^{3/2}(\partial\Omega)$ \cite[Theorem 1 (iii)]{Mazya-Shaposhnikova05}, we have $u(x)= {\SL}[\psi](x)$ for $x \in \GO^+$ if we set $\psi:= \SL^{-1}[u|_{+}]$ on $\p\GO$. Thus we have the following characterization of $\Mpm^+$:
\begin{equation}\label{201}
        \Mpm^+ =\{ \nabla {\SL}[\psi]|_+ \mid \psi\in H^{1/2} (\partial\Omega)\}.
    \end{equation}
Likewise, we have
\begin{equation}\label{202}
        \Mpm^- =\{ \nabla {\SL}[\psi]|_- \mid \psi\in H^{1/2} (\partial\Omega)\}.
    \end{equation}

We are now ready to prove Theorem \ref{1001}.

\begin{proof}[Proof of Theorem \ref{1001}]
For any subset $Y$ of $H^{1/2}(\p\GO, \Rbb^3)$, we define the annihilator $Y^0$ by
    \[
            Y^0 := \{ f \in H^{-1/2}(\p\GO, \Rbb^3) \mid \langle f, g \rangle_{\p\GO}=0, \ \forall g\in Y \}.
        \]
     Since $(\Ker \Acal_\pm)^0=\overline{\Ran \Acal_\pm^*}$, we infer from Theorems \ref{theo_apm_div_free} and \ref{theo_ran_apms_closed} that
     \[
     (\Hdf^\pm)^0=\Ran \Acal_\pm^*.
     \]

    By the definition \eqref{eq_inner_product} of the inner products on $H^{1/2}(\partial\Omega)$ and $H^{1/2}(\partial\Omega, \Rbb^3)$, we have
    \[
        f\in (\Hdf^\pm)^\perp \Longleftrightarrow \jbk{\SL^{-1}[f], g}_{\partial\Omega}=0, \ \forall g\in \Hdf^\pm  \Longleftrightarrow \SL^{-1}[f]\in (\Hdf^\pm)^0,
    \]
    which implies $(\Hdf^\pm)^\perp={\SL}[(\Hdf^\pm)^0]$. Thus we have
    \[
        (\Hdf^\pm)^\perp= {\SL}[\Ran \Acal_\pm^*]=\Ran \SL \Acal_\pm^*.
    \]
    By \eqref{eq_jump_Js} and \eqref{eq_aps_j}, we have
    \[
    \SL \Acal_\pm^*[\varphi]= - \nabla {\SL}[\varphi]|_\mp,
    \]
    and hence
    \[
    \Ran \SL \Acal_\pm^* =\{ \nabla {\SL}[\psi]|_\mp \mid \psi\in H^{1/2} (\partial\Omega)\}.
    \]
    Thus \eqref{1000} follows from \eqref{201} and \eqref{202}.
\end{proof}

Before finishing this section, we prove the following theorem which will be used in the next section.

\begin{theo}
    \label{theo_dim_sum_below}
    Suppose that $\partial\Omega$ is $C^{1, \alpha}$ for some $\alpha>1/2$. Let
    \begin{equation}\label{Xcal_def}
    \Xcal^\pm:=\Hdrf^\pm \cap \Hdf.
    \end{equation}
    It holds that
    \begin{align}
        \Hdrf^\pm &=\Mpm^\pm \oplus \Xcal^\pm, \label{eq_decomp_drf} \\
        \Hdf &=\Hdf^\Dcal \oplus \Xcal^- \oplus \Xcal^+. \label{eq_decomp_drf2}
    \end{align}
\end{theo}

\begin{proof}
We only prove \eqref{eq_decomp_drf} for $\Hdrf^+$. The inclusion $\supset$ is obvious. To prove the opposite inclusion, suppose $f\in \Hdrf^+$. According to Theorem \ref{theo_vf_potential}, $f$ admits an orthogonal decomposition $f=g_- + g_+ +h$ where $g_\pm \in \Mpm^\pm$ and $h\in \Hdf$. Thus we have
\[
\jbk{f, g_-}_* = \jbk{g_-, g_-}_*.
\]
Since $\Mpm^- \subset \Hdrf^-$ and $\Hdrf^- \perp \Hdrf^+$, we have $\jbk{f, g_-}_*=0$, and hence $g_-=0$. Since $\Mpm^+ \subset \Hdrf^+$, we have
    \[
        h=f-g_+ \in\Hdrf^+.
    \]
Thus $h\in \Hdrf^+\cap \Hdf=\Xcal^+$ and hence $f=g_+ +h \in \Mpm^+ + \Xcal^+$. So, $\Hdrf^+ \subset \Mpm^+ \oplus \Xcal^+$.

The decomposition \eqref{eq_vf_decomposition2} together with \eqref{eq_decomp_drf} yields
\[
H^{1/2}(\partial\Omega, \Rbb^3)=\Mpm^- \oplus \Xcal^- \oplus \Mpm^+ \oplus \Xcal^+ \oplus \Hdf^\Dcal,
\]
from which \eqref{eq_decomp_drf2} follows.
\end{proof}

\section{Proof of Theorem \ref{theo_betti_from_below}}
\label{subs_Betti}

In order to relate the spaces $\Hdf^\pm/\Mpm^\pm$ to the Betti numbers of $\GO^\pm$, we heavily use results of \cite{Osterbrink-Pauly20, Picard82}. So, we begin by recalling them.

Let $U\subset \Rbb^3$ be a (possibly unbounded) Lipschitz domain with the compact boundary $\partial U$. For our purpose, $U$ is either $\GO^-$ or $\GO^+$. We set
\begin{align*}
    H(\rmop{rot}, U):=&\,\{ u\in L^2 (U, \Rbb^3) \mid \nabla\times u\in L^2 (U, \Rbb^3)\}, \\
    H(\rmop{rot}0, U):=&\,\{ u\in L^2 (U, \Rbb^3) \mid \nabla\times u=0 \text{ in } U\}.
\end{align*}
Let $H_0(\mathrm{div}, U)$ be the completion of the space $C_c^\infty (U, \Rbb^3)$ by the graph norm
\[
    u\in C_c^\infty (U, \Rbb^3)\longmapsto (\|u\|_{L^2 (U, \Rbb^3)}^2+\| \nabla\cdot u\|_{L^2 (U)}^2)^{1/2}.
\]
Roughly speaking, $H_0(\mathrm{div}, U)$ is the space of all $u\in L^2 (U, \Rbb^3)$ such that $\nabla\cdot u\in L^2 (U)$ and $\nv\cdot u=0$ on $\partial U$.
Finally, we define
\[
    \Hscr_\mathrm{t} (U):=\{u\in H(\rmop{rot}0, U)\cap H_0 (\rmop{div}, U) \mid \nabla\cdot u=0 \text{ in } U\}.
\]
(Here the subscript $\mathrm{t}$ stands for ``tangential'', which means the boundary condition $\nv\cdot u=0$.)

We use the following theorem which summarizes the results from \cite{Osterbrink-Pauly20, Picard82} (We also refer to \cite{Kress72, Picard79} for related works). Here, $\nabla H^1 (\Omega)$ denotes the space of all $\nabla u$ for $u \in H^1 (\Omega)$, and $\nabla H^1_{-1} (\Omega^+)$ does likewise.

\begin{theo}\label{theo_Helmholtz}
    Let $\Omega\subset \Rbb^3$ be a bounded Lipschitz domain.
    \begin{enumerate}[label=(\roman*)]
        \item \label{enum_Helmholtz_decomposition}\cite{Osterbrink-Pauly20} The following orthogonal decompositions with respect to the $L^2$-inner products hold:
        \begin{align}
            H(\rmop{rot}0, \Omega^-)&=\nabla H^1 (\Omega^-)\oplus \Hscr_\mathrm{t}(\Omega^-), \label{eq_Helmholtz_interior}\\
            H (\rmop{rot}0, \GO^+)&=\nabla H^1_{-1} (\GO^+)\oplus \Hscr_\mathrm{t} (\GO^+). \label{eq_Helmholtz_exterior}
        \end{align}
        \item \label{enum_harmonic_Betti}\cite{Picard82} The following holds:
        \begin{equation}\label{dimb1}
        \dim \Hscr_\mathrm{t}(\GO^\pm)=b_1(\GO^\pm).
        \end{equation}
    \end{enumerate}
\end{theo}

Note that \eqref{eq_Helmholtz_interior} and \eqref{eq_Helmholtz_exterior} are decompositions of Helmholtz type since rotation-free fields are decomposed into gradient fields and rotation-free, divergence-free fields.

We define the operators $\Gcal_\pm$ as follows:
\begin{align}
    &\Gcal_-: \overline{f} \in \Hdrf^-/\Mpm^- \mapsto \overline{v_-^f} \in H(\rmop{rot}0, \Omega^-)/\nabla H^1 (\Omega^-),  \label{eq_isomorphism_interior}\\
    &\Gcal_+: \overline{f} \in \Hdrf^+/\Mpm^+ \mapsto \overline{v_+^f} \in H(\rmop{rot}0, \GO^+)/\nabla H^1_{-1} (\GO^+). \label{eq_isomorphism_exterior}
\end{align}
Here $\overline{f}$ and $\overline{v_\pm^f}$ denote equivalence classes in the corresponding quotient spaces. Note that $\Gcal_\pm$ are well-defined. In fact, if $f \in \Hdrf^+$, then $v_+^f \in H(\rmop{rot}0, \GO^+)$. If $\overline{f}=\overline{g}$, then $f-g \in \Mpm^+$, i.e., there is a harmonic function $u$ in $\GO^+$ belonging to $H^2_{-1} (\GO^+)$ such that $f-g= \nabla u|_+$ on $\p\GO$. Thus $v_+^f-v_+^g= \nabla u$ in $\GO^+$, and hence $\overline{v_\pm^f}= \overline{v_\pm^g}$. Thus $\Gcal_+$ is well-defined. One can see that $\Gcal_-$ is well-defined similarly.

We need the following lemma to prove Theorem \ref{theo_betti_from_below}.
\begin{lemm}\label{lemm_regularity}
    If $\Omega\subset \Rbb^3$ is a bounded $C^{1, \alpha}$-domain for some $\alpha>1/2$, then 
    \begin{equation}\label{H1incl}
        H(\rmop{rot}, U)\cap H_0(\rmop{div}, U) \subset H^1 (U, \Rbb^3)
    \end{equation}
    for $U=\GO^\pm$.
\end{lemm}

\begin{proof}
    The case when $U=\Omega^-$ is proved in \cite{Filonov97}. That when $U=\Omega^+$ is proved in \cite{Kuhn-Pauly10} under the assumption that $\partial\Omega$ is $C^2$. When $U=\Omega^+$ and $\partial\Omega$ is $C^{1, \alpha}$ for some $\alpha>1/2$, we consider a $C^\infty$-smooth cutoff function
    \[
        \chi (x)=
        \begin{cases}
            1 & \text{if } x\in B_{2R}, \\
            0 & \text{if } x\in \Rbb^3\setminus B_{3R}
        \end{cases}
    \]
    where $B_r$ is the open ball in $\Rbb^3$ centered at $0$ with radius $r>0$, and we fix $R>0$ such that $\overline{\Omega}\subset B_R$. Then we decompose a vector field $u\in H(\rmop{rot}, \Omega^+)\cap H_0 (\rmop{div}, \Omega^+)$ into $u=\chi u+(1-\chi)u$. Since
    \[
        \chi u\in H(\rmop{rot}, B_{4R}\setminus \overline{\Omega})\cap H_0(\rmop{div}, B_{4R}\setminus \overline{\Omega})
        \subset H^1 (B_{4R}\setminus \overline{\Omega})
    \]
    and
    \[
        (1-\chi)u\in H (\rmop{rot}, \Rbb^3\setminus \overline{B_R})\cap H_0 (\rmop{div}, \Rbb^3\setminus \overline{B_R})\subset H^1(\Rbb^3\setminus \overline{B_R}),
    \]
    we obtain $u\in H^1 (\Omega^+)$.
\end{proof}

\begin{rema*}
    For general bounded Lipschitz domains, it is known that the space $H(\rmop{rot}, \Omega^\pm)\cap H_0 (\rmop{div}, \Omega^\pm)$ is included in the Sobolev space $H^{1/2}(\Omega^\pm, \Rbb^3)$ \cite{Costabel90, MMT01}. This Sobolev regularity is optimal in the sense that there is a bounded $C^1$-domain $\Omega$ and a vector field $u\in H(\rmop{rot}, \Omega)\cap H_0 (\rmop{div}, \Omega)$ such that $u\not\in H^{1/2+\varepsilon}(\Omega, \Rbb^3)$ for any $\varepsilon>0$ \cite[Proposition 1.1]{Costabel19}. 
\end{rema*}

\begin{proof}[Proof of Theorem \ref{theo_betti_from_below}]
We see from \eqref{eq_decomp_drf} and \eqref{eq_decomp_drf2} that
\[
        \dim (\Hdf/\Hdf^\Dcal)
        =\dim (\Hdrf^-/\Mpm^-)+\dim (\Hdrf^+/\Mpm^+).
\]
Since $b_1(\GO^-) + b_1(\GO^+)=b_1(\partial\Omega)$ (see \cite[Proof of Theorem 1.2]{FJKc1a} for a proof of this identity), it suffices to prove \eqref{eq_betti_from_below1}. We prove it by proving bijectivity of $\Gcal_\pm$ defined in \eqref{eq_isomorphism_interior} and \eqref{eq_isomorphism_exterior}. We only deal with the case of $\Gcal_+$ since the case for $\Gcal_-$ can be dealt with similarly.

Suppose that $f \in \Hdrf^+$ and $\Gcal_+[\overline{f}]=\overline{v_+^f}=0$. Then, $v_+^f=\nabla u$ in $\GO^+$ for some $u\in H^1_{-1}(\GO^+)$. Since $v_+^f \in H^1_{-1}(\GO^+)$, $u$ actually belongs to $H^2_{-1}(\GO^+)$.
Since $f\in \Hdrf^+$, $u$ is harmonic in $\GO^+$. Hence we have $f\in \Mpm^+$. This implies $\overline{f}=0$.

To prove that $\Gcal_+$ is surjective, suppose $u \in H (\rmop{rot}0, \GO^+)$. According to \eqref{eq_Helmholtz_exterior}, there is a unique $v \in \Hscr_\mathrm{t} (\GO^+)$ such that $\overline{u}=\overline{v}$. The Green theorem shows that $\jbk{v, \lap w}_{L^2 (\GO^+)}=0$ for all $w\in C_c^\infty (\GO^+)$. Thus $\lap v=0$ on $\GO^+$ in the classical sense. Moreover, $v \in H^1(\GO^+)$ by Lemma \ref{lemm_regularity}.  Thus $f:=v|_+$ belongs to $H^{1/2}(\p\GO)$ by trace theorem, and $v_+^f=v$ by uniqueness of the solution to the problem \eqref{eq_bvp_exterior}. Thus $\Gcal_+[\overline{f}]=\overline{u}$. So $\Gcal_+$ is surjective.
\end{proof}

\section*{Discussion}

The decompositions \eqref{eq_vf_decomposition} and \eqref{eq_vf_potential} show that if $f \in H^{1/2}(\p\GO, \Rbb^3)$, then there are $f_- \in \Hdrf^-$ and $f_+ \in \Hdrf^+$,  $f_0 \in \Ran \left( \NP^\divf + 1/2I \right)\left( \NP^\divf - 1/2I \right)$,  harmonic functions $h_- \in H^2(\GO^-)$ and $h_+ \in H^2_{-1}(\GO^+)$, and $g_0 \in \Hdf$ such that
\[
f=f_- + f_+ + f_0= \nabla h_-|_- + \nabla h_+|_+ + g_0.
\]
If $\p\GO$ is simply connected, then
\[
f_- = \nabla h_-|_, \quad f_+ =\nabla h_+|_+ , \quad f_0= g_0.
\]
It is interesting to find a way to construct these functions for a given $f$.

Theorems \ref{1001} and \ref{theo_vf_potential} are valid for bounded domains $\GO$ in $\Rbb^d$ with $d >3$ provided that $\p\GO$ is $C^{1, 1}$-smooth. This can proved in the same way as the three-dimensional case with minor modification. We only mention differences of the proof from the three-dimensional case. On $\Rbb^d$ with $d > 3$, the equality $\nabla \times u=0$ means $\p u_k/\p x_j=\p u_j/\p x_k$ for all $j, k=1, \ldots, d$, and the fundamental solution $\GG(x)$ is given by $-((d-2)\Go_d)^{-1} |x|^{2-d}$ where $\Go_d$ is the area of the unit sphere in $\Rbb^d$. The $C^{1, 1}$-regularity condition of the boundary is needed for the higher-dimensional analogue of the inclusion relation \eqref{H1incl}, i.e., the $H^1$-regularity of $H(\rmop{rot}, \Omega^\pm)\cap H_0(\rmop{div}, \Omega^\pm)$. The result of \cite{Filonov97}, which is used for the proof of Lemma \ref{lemm_regularity}, holds only in the three-dimensional case. However, if $\p\GO$ is $C^{1, 1}$-smooth, then it is proved in \cite{Morrey56} that 
\[
H(\rmop{rot}, \Omega^-)\cap H_0(\rmop{div}, \Omega^-) \subset H^1 (\GO^-, \Rbb^d)
\]
even if $d >3$. For $\Omega^+$, since the result in \cite{Kuhn-Pauly10}, which we utilize in the proof of Lemma \ref{lemm_regularity}, is available in higher dimensional cases, the same argument as in Lemma \ref{lemm_regularity} proves the corresponding assertion. The case when the boundary is $C^2$-smooth, the above inclusion relation is an immediate consequence of the Friedrichs inequality, which is also called Gaffney inequality \cite{Friedrichs55, Gaffney51} (see also \cite[Theorem 3.1]{Kuhn-Pauly10} for anisotropic electromagnetic media case and \cite{Mitrea01} for convex domain cases which may contain singularities). It is unknown to the best of the authors' knowledge whether the boundary regularity can be weakened to $C^{1, \alpha}$ for some $\alpha>1/2$.



\begin{thebibliography}{99}
    \bibitem{Bott-Tu82} R. Bott and L. W. Tu. \textit{Differential forms in algebraic topology}, volume 82 of \textit{Graduate Texts in Mathematics}. Springer-Verlag, New York-Berlin, 1982.

    \bibitem{Costabel90} M. Costabel. A remark on the regularity of solutions of {M}axwell's equations on {L}ipschitz domains. \textit{Math. Methods Appl. Sci.}, 12(4):365--368, 1990.

    \bibitem{Costabel19} M. Costabel. On the limit Sobolev regularity for Dirichlet and Neumann problems on Lipschitz domains. \textit{Math. Nachr.}, 292(10):2165--2173, 2019.

    \bibitem{DKV88} B. Dahlberg, C. E. Kenig, and G. Verchota. Boundary value problems for
    the systems of elastostatics in Lipschitz domains. \textit{Duke Math. J.}, 57(3):795--818, 1988.

    \bibitem{Filonov97} N. Filonov. Principal singularities of the magnetic field component in resonators with a boundary of a given class of smoothness. \textit{Algebra i Analiz}, 9(2):241--255, 1997.

    \bibitem{Friedrichs55} K. O. Friedrichs. Differential forms on {R}iemannian manifolds. \textit{Comm. Pure Appl. Math.}, 8:551--590, 1955.

    \bibitem{FJKc1a} S. Fukushima, Y.-G. Ji, and H. Kang. A decomposition theorem of surface vector fields and spectral properties of the Neumann-Poincar\'e operator in elasticity. arXiv:2211.15879v2 [math.AP], to appear in \textit{Trans. Amer. Math. Soc}.

    \bibitem{Gaffney51} M. P. Gaffney. The harmonic operator for exterior differential forms. \textit{Proc. Nat. Acad. Sci. U.S.A.}, 37:48--50, 1951.

    \bibitem{Gesztesy-Mitrea11} F. Gesztesy and M. Mitrea. A description of all self-adjoint extensions of
    the Laplacian and Kre\u{\i}n-type resolvent formulas on non-smooth domains. \textit{J. Anal. Math.}, 113:53--172, 2011.

    \bibitem{GMM11} V. Gol'dshtein, I. Mitrea, and M. Mitrea. Hodge decompositions with mixed boundary conditions and applications to partial differential equations on Lipschitz manifolds. \textit{J. Math. Sci. (N.Y.)}, 172(3):347--400, 2011. In \textit{Problems in mathematical analysis. No. 52}.

    \bibitem{HKSSY212D} M. Hieber, H. Kozono, A. Seyfert, S. Shimizu, and T. Yanagisawa. The Helmholtz--Weyl decomposition of $L^r$ vector fields for two dimensional exterior domains. \textit{J. Geom. Anal.}, 31:5146--5165, 2021.

    \bibitem{HKSSY213D} M. Hieber, H. Kozono, A. Seyfert, S. Shimizu, and T. Yanagisawa. $L^r$-Helmholtz-Weyl decomposition for three dimensional exterior domains. \textit{J. of Funct. Anal.}, 281(8):109144, 2021.

    \bibitem{Karpukhin19} M. A. Karpukhin. The Steklov problem on differential forms. \textit{Canad. J. Math.}, 71(2):417--435, 2019.

    \bibitem{KPS} D. Khavinson, M. Putinar, and H. S. Shapiro. Poincar\'e's variational problem in potential theory. \textit{Arch. Rational Mech. Anal.}, 185(1):143--184, 2007.

    \bibitem{Kress72} R. Kre\ss. Potentialtheoretische Randwertprobleme bei Tensorfeldern beliebiger Dimension und beliebigen Ranges. \textit{Arch. Rational Mech. Anal.}, 47:59--80, 1972.

    \bibitem{Kuhn-Pauly10} P. Kuhn and D. Pauly. Regularity results for generalized electro-magnetic problems. \textit{Analysis (Munich)}, 30(3):225--252, 2010.

    \bibitem{Mazya-Shaposhnikova05} V. Maz'ya and T. Shaposhnikova. Higher regularity in the layer potential theory for Lipschitz domains. \textit{Indiana Univ. Math. J}, 54(1):99--142, 2005.

    \bibitem{Mitrea01} M. Mitrea. Dirichlet integrals and {G}affney-{F}riedrichs inequalities in convex domains. \textit{Forum Math.}, 13(4):531--567, 2001.

    \bibitem{MMT01} D. Mitrea, M. Mitrea, and M. Taylor. Layer potentials, the {H}odge {L}aplacian, and global boundary problems in nonsmooth {R}iemannian manifolds. \textit{Mem. Amer. Math. Soc.}, 150(713), 2001.

    \bibitem{Morrey56} C. B. Morrey, Jr. A variational method in the theory of harmonic integrals. II. \textit{Amer. J. Math.}, 78:137--170, 1956.

    \bibitem{Osterbrink-Pauly20} F. Osterbrink and D. Pauly. Low frequency asymptotics and electro-
    magneto-statics for time-harmonic Maxwell's equations in exterior weak Lipschitz domains with mixed boundary conditions. \textit{SIAM J. Math. Anal.}, 52(5):4971--5000, 2020.

    \bibitem{Picard79} R. Picard. Zur {T}heorie der harmonischen {D}ifferentialformen, \textit{Manuscripta Math.}, 27(1):31--45,
    1979.

    \bibitem{Picard82} R. Picard. On the boundary value problems of electro- and magnetostatics. \textit{Proc. Roy. Soc. Edinburgh Sect. A}, 92(1-2):165--174, 1982.

    \bibitem{Schwarz95} G. Schwarz. \textit{Hodge decomposition---a method for solving boundary value problems}, \textit{Lecture Notes in Mathematics}, Vol. 1607, Springer-Verlag, Berlin, 1995.

    \bibitem{Verchota84} G. Verchota. Layer potentials and regularity for the Dirichlet problem for Laplace's equation in Lipschitz domains, \textit{J. Funct. Anal.} 59(3):572--611, 1984.
\end{thebibliography}

\end{document}